\documentclass[12pt]{article}

\usepackage{amsmath,amssymb,showlabels}

\def \F{I\!\!F}
\def \H{I\!\!H}

\def \R{I\!\!R}

\def\Fc{{\cal F}}

\def\Dzw1#1{\frac{\partial^2 #1}{\partial z \partial w_1}}

\def\Dzb1#1{\frac{\partial^2 #1}{\partial z \partial b_1}}

\numberwithin{equation}{section}

\usepackage{amsthm}
\newtheorem{definition}{Definition}[section]

\newtheorem{theorem}[definition]{Theorem}

\newtheorem{remark}[definition]{ \it Remark}
\newtheorem{coro}[definition]{Corollary}
\newtheorem{proposition}[definition]{Proposition}
\newtheorem{lemma}[definition]{Lemma}

\def\R{{\bf R}}

\def\1B{\text{1\!\!I}}

\def\R{{\bf R}}

\def\1B{\text{1\!\!I}}

\def\vp{\stackrel{\vee}{\pi}}

\def\p{\partial}

\def\v{\varphi}

\def\g{\gamma}
\def\s{\sigma}
\def\a{\alpha}
\def\b{\beta}

\def\d{\delta}
\def\vp{\varepsilon}

\def\cal#1{\mathcal{#1}}

\def \H{\mathbb {H}}
\def \R{\mathbb {R}}

\def\dd{\displaystyle}

\begin{document}

\title{Optimal stopping for dynamic risk measures with jumps and obstacle problems}
\author{ Roxana DUMITRESCU\thanks{CEREMADE,
Universit\'e Paris 9 Dauphine, CREST  and  INRIA Paris-Rocquencourt, email: \textbf{roxana@ceremade.dauphine.fr}} \and Marie-Claire Quenez\thanks{LPMA, 
Universit\'e Paris 7 Denis Diderot, Boite courrier 7012, 75251 Paris Cedex 05, France, 
email: {\tt quenez@math.univ-paris-diderot.fr}}
\and 
Agn\`es SULEM
\thanks{ INRIA Paris-Rocquencourt, Domaine de Voluceau, Rocquencourt, BP 105, Le Chesnay Cedex, 78153, France, and Universit\'e Paris-Est, F-77455 Marne-la-Vall\'ee, France, email: {\tt agnes.sulem@inria.fr}}}

\date{\today}

\maketitle

\begin{abstract}

We  study  the optimal stopping problem for a monotonous dynamic risk measure  induced by a BSDE with jumps in the Markovian case. We show that the value function is a  viscosity solution  of an obstacle problem   for a 
partial integro-differential variational inequality, 
and we provide an uniqueness result for this obstacle problem.
\end{abstract}



\newpage
\section{Introduction}

In the last years, there has been  several studies on dynamic risk measures and their links with nonlinear  backward stochastic differential equations (BSDEs). We recall that nonlinear BSDEs have been introduced  in  \cite{14}   in a  Brownian framework,  in order to provide a probabilistic representation of semilinear parabolic partial-differential equations. BSDEs with jumps and their links with partial integro-differential equations  are studied  in \cite{2}. A comparison theorem is established in \cite{18}  and generalized in \cite{16}, where properties of  dynamic risk measures induced by  BSDEs with jumps are also provided. 
  An optimal stopping problem for  such risk measures is addressed in  \cite{17},
and  the value function is characterized as the solution of a reflected BSDE with jumps and RCLL obstacle process.

 In the present paper, we  focus on the optimal stopping problem for dynamic risk measures induced by BSDEs with jumps  in a Markovian framework. In this case the driver of the BSDE depends on a given state process
 $X$, which can represent, for example, an index or  a stock  price. This process will be assumed to be driven by a Brownian motion and a Poisson random measure.

 Our main contribution consists in establishing the link between the value function of our optimal stopping problem and  parabolic partial integro-differen\-tial variational inequalities (PIDVIs).  We prove that the minimal risk measure, which corresponds to the solution of a reflected BSDE with jumps, is a viscosity solution of a PIDVI. This provides an existence result for the obstacle problem under relatively 
 weak assumptions. 
In the Brownian case, this result was obtained   in  \cite{10} by using a penalization method via non-reflected BSDEs.
 Note that this method could also be adapted to our  case with jumps,  but  would involve heavy computations in order to  prove the convergence of the solutions of the penalized BSDEs to the solution of the reflected BSDE. It would also   require some  convergence results of the viscosity solutions theory in the integro-differential case.
  We provide here instead a direct and shorter proof.

Furthermore,  under some  additional assumptions, we prove a comparison theorem in the class of bounded continuous functions, relying on a 
non-local version of Jensen-Ishii Lemma (see \cite{1}), from which the uniqueness of the viscosity solution follows. 
We point out that our problem is not covered by the study in \cite{1},  since we are dealing with nonlinear BSDEs, and this  leads to a more complex integro-differential 
operator in the associated PDE.

 The paper is organized as follows: In Section \ref{sec2}  we give the formulation of our optimal stopping problem. In Section  \ref{sec3}, we prove that the value function is  a solution of an obstacle problem for a PIDVI in the viscosity sense. In  Section  \ref{sec4},  we establish an uniqueness result.  In the Appendix, we prove  some  estimates, from which we derive that  the value function is continuous and has polynomial growth and provide some  complementary results.

\section{Optimal Stopping Problem for Dynamic Risk Measures with Jumps in the Markovian Case}\label{sec2}

Let $(\Omega,  \Fc_, P)$ be a probability space.
Let  $W$ be   a  one-dimensional Brownian motion and  $N(dt,du)$ be a Poisson random measure with compensator $\nu(du)dt$ such that
$\nu$ is a $\sigma$-finite measure on $\R^*$ equipped with its Borel field $ \bf {\cal B}(\R^*),$ and satisfies $\int_{\mathbb{R}^*}(1 \wedge e^2) \nu(de) < \infty$.
 Let $\tilde N(dt,du)$ be its compensated process.
Let  $\F = \{\Fc_t , t \geq 0 \}$
be  the natural filtration associated with $W$ and $N$.

We consider a state process $X$ which may be interpreted 
as an index, an interest rate process, an economic factor, an indicator of the market or  the value of a portfolio, which has an influence on the risk measure and the position.
%
For each initial time $t \in [0,T]$ and each condition $x \in \mathbb{R}$, let $X^{t,x}$ be the solution of the following stochastic differential equation (SDE):
\begin{equation}\label{SDE}
X_s^{t,x}=\dd x+ \int_t^s b(X_r^{t,x})dr+\int_t^s \sigma(X_r^{t,x})dW_r+\int_t^s \int_{\mathbb{R}^*} \beta(X_{r^{-}}^{t,x},e) \Tilde{N}(dr,de),
\end{equation}
where
$b, \sigma :\mathbb{R} \rightarrow \mathbb{R}$  are   Lipschitz continuous,  and
$\beta : \mathbb{R} \times\mathbb{R}^* \rightarrow \mathbb{R}$  is a measurable function  such that for some non negative real $C$, and for all $e \in \mathbb{R}$
\begin{align*}
&|\beta(x,e)| \leq C(1 \wedge |e|),  \;\; x \in \mathbb{R}\\
&|\beta(x,e)- \beta(x',e)| \leq C|x-x'|(1 \wedge |e|), \;\;  x, x'\in \mathbb{R}.
\end{align*}

We  introduce a dynamic risk measure $\rho$ induced by a BSDE with jumps. For this, we consider two functions
 $\gamma$ and $f$  satisfying the following assumption:\\
 
{\bf Assumption $\textbf{H}_1$}
\begin{itemize}
\item[$\bullet$]
$\gamma: \mathbb{R} \times \mathbb{R}^* \rightarrow \mathbb{R}$ is $ \mathcal{B}(\mathbb{R}) \otimes \mathcal{B}(\mathbb{R}^*)$-measurable,   \\
$|  \gamma(x,e)-\gamma(x',e)| < C|x-x'|(1 \wedge |e|), x, x' \in \mathbb{R}, e \in \mathbb{R}^*$\\
  $-1 \leq \gamma(x,e) \leq C(1 \wedge |e|)$, $e \in \mathbb{R}^*$
\item[$\bullet$]  $f: [0,T] \times  \mathbb{R}^3 \times  L_{\nu}^2  \rightarrow \mathbb{R}$
is continuous in $t$ uniformly with respect to $x,y,z,k$,  and continuous in $x$ uniformly  with respect to $y,z,k$.

\begin{itemize}

\item[(i)]
 $|f(t,x,0,0,0)| \leq C(1+x^p), \;  \forall x\in\mathbb{R}$

\item[(ii)]
$ |f(t,x,y,z,k)- f (t,x',y',z',k')| \leq C(|y-y'|+|z-z'|+\|k-k'\|_{L_{\nu}^2})$, $\forall \;  t \in [0, T]$, $y, y', z,z' \in \mathbb{R}$, $k,k' \in L_{\nu}^2$

\item[(iii)]
$ f(t,x,y,z,k_1)- f (t,x,y,z,k_2) \geq <\gamma(x,\cdot),k_1-k_2>_{\nu}, \forall t,x,y,z,k_1,k_2.$
\end{itemize}

\end{itemize}

Here, $L_{\nu}^2$ denotes  the set of Borelian functions $\ell: \R^* \rightarrow \R$ such that  $\|\ell\|_{\nu}^2 :=\int_{\R^*}  |\ell(u) |^2 \nu(du) < + \infty.$
It is a 
 Hilbert space equipped with the scalar product 
$\langle \delta    , \, \ell \rangle_\nu := \int_{\mathbb{R}^*} \delta(e) \ell(e) \nu(de)$ for all $  \delta  , \, \ell \in {L}^2_\nu \times {L}^2_\nu.$ 

We also  introduce the set
 $\mathbb{H}^{2}$ (resp. $\mathbb{H}^{2}_{\nu}$)  of predictable processes ($\pi_t$) (resp. $(l_t(\cdot))$) such that $\mathbb{E}\int_0^T \pi_s^2ds {<} \infty$ (resp. $\mathbb{E}\int_0^T \!\|l_s\|_{L_{\nu}^2}^2ds{<}\infty$); the set
  $\mathcal{S}^2$ of real-valued RCLL adapted processes $(\varphi_s)$  with $\mathbb{E}[\sup_s \varphi_s^2] < \infty$, and the set
  $L^2({\cal F}_T)$    of $\Fc _T$-measurable and square-integrable random variables.

Let $(t,x)$ be a fixed intial condition. For each maturity $S$  in $[t,T]$ and each position $\zeta$  in $L^2({\mathcal{F}_S})$, the associated risk measure at time $s \in [t,S]$ is defined by
\begin {equation}
\rho_s^{t,x}(\zeta, S):=-\mathcal{E}_{s,S}^{t,x}(\zeta),  \; t \leq s \leq S,
\end{equation}
where $\mathcal{E}^{t,x}_{\cdot,S}(\zeta)$ denotes the $f$-conditional expectation, starting at $(t,x)$, defined as the solution in $\mathcal{S}^{2}$
 of the BSDE with Lipschitz driver $f(s, X_s^{t,x},y,z,k)$, terminal condition $\zeta$ and terminal time $S$, that is the solution $(\mathcal{E}_s^{t,x})$ of
\begin{equation}\label{2.2}
-d\mathcal{E}_s=f(s, X_s^{t,x},\mathcal{E}_s,\pi_s,l_s(\cdot))ds-\pi_sdW_s-\int_\mathbb{R^*} l_s(u) \Tilde{N}(dt,du)\,; \,\,\mathcal{E}_S=\zeta,
\end{equation}
where $(\pi_s)$, $(l_s)$ are the associated processes, which belong to $\mathbb{H}^{2}$ and $\mathbb{H}^{2}_{\nu}$ respectively. \\
The functional $\rho:(\zeta,S) \rightarrow \rho_{\cdot}(\zeta,S)$ defines then a dynamic risk measure induced by the BSDE with driver $f$ (see \cite{16}).
Assumption $\textbf{H}_1$ implies that the driver $f(s, X_s^{t,x},y,z,k)$ satisfies Assumption 3.1 in \cite{17}, which ensures the monotonocity property of $\rho$ with respect to $\zeta$. 
More precisely, for each maturity $S$ and  for each positions $\zeta_1$, $\zeta_2 \in {L}^2(\mathcal{F}_S)$, with $\zeta_1 \leq \zeta_2$ a.s., we have
$\rho_s^{t,x}(\zeta_1, S) \geq \rho_s^{t,x}(\zeta_2, S)$ a.s.

 We now formulate our  optimal stopping problem for dynamic risk measures. For each $(t,x) \in [0,T] \times \mathbb{R}$, we consider a dynamic financial position  given by the process $\left(\xi_s^{t,x}, t \leq s \leq T \right)$, defined  via the state process $(X_s^{t,x})$ and two  functions $g$ and $h$ such that
\begin{itemize}
\item[$\bullet$]
$g \in \mathcal{C}(\mathbb{R})$ with at most polynomial growth at infinity,

\item[$\bullet$]
$
h:[0,T] \times \mathbb{R} \rightarrow \mathbb{R}
$
is  continuous in $t$, $x$ and there  exist $p\in \mathbb{N}$ and a real constant  $C$, such that
\begin{equation}\label{4.3}
 |h(t,x)| \leq C(1+ |x|^p), \forall t \in [0,T], x \in \mathbb{R}, 
\end{equation}

\item[$\bullet$]
$h(T,x) \leq g(x),  \;\; \forall x \in \mathbb{R}.$
\end{itemize}
For each initial condition $(t,x) \in [0,T] \times \mathbb{R}$, the dynamic position is then  defined by:
\begin{eqnarray*}
\begin{cases}
\xi^{t,x}_s :=h(s,X_s^{t,x}), \,\,\, s <T\\
\xi^{t,x}_T :=g(X^{t,x}_T).\\
\end{cases}
\end{eqnarray*}

Let $t \in [0,T]$ be the initial time and let $x \in \mathbb{R}$ be the initial condition. 
The minimal risk measure at time $t$ is  given by:
\begin{equation}\label{vs}
{\rm ess} \inf_{\tau \in \mathcal{T}_t} \rho_t^{t,x} (\xi_{\tau}^{t,x}, \tau)=- {\rm ess} \sup_{\tau \in \mathcal{T}_t} \mathcal{E}_{t, \tau}^{t,x}(\xi_\tau^{t,x}).
\end{equation}
Here $\mathcal{T}_t$ denotes the set of stopping times with values in $[t,T]$.

By Th. 3.2 in \cite{17}, the minimal risk measure  is  characterized via the solution $Y^{t,x}$  in $\mathcal{S}^2$ of the following reflected BSDE  (RBSDE) associated with driver $f$ and obstacle $\xi$:
\begin{equation}\label{markRBSDE}
\begin{cases}
Y^{t,x}_s=g(X^{t,x}_T)+\dd\int_{s}^{T}f(r,X^{t,x}_r,Y^{t,x}_r,Z^{t,x}_r, K^{t,x}_r(\cdot))dr+ A^{t,x}_T-A^{t,x}_s\\\qquad\dd-\int_s^T Z^{t,x}_rdW_r-\int_s^T \dd\int_{\mathbb{R}^*}K^{t,x}(r,e) \Tilde{N}(dr,de)\\
Y^{t,x}_s \geq \xi^{t,x}_s, 0 \leq s \leq T\; \text{ a.s. }\\
A^{t,x}  \text{ is a nondecreasing, continuous predictable process in $\mathcal{S}^2$ with } \\\qquad \qquad A^{t,x}_t=0 \text{ and such that }\\
\dd\int_{t}^{T}(Y^{t,x}_s-\xi^{t,x}_s)dA^{t,x}_s=0 \text{ a.s. \,,  }
\end{cases}
\end{equation}
with $Z^{t,x}, K^{t,x} \in \mathbb{H}^2$ (resp. $\mathbb{H}_{\nu}^2$).
Note that by the  assumptions made on $h$ and $g$, the obstacle $(\xi_s^{,t,x})_{s \geq t}$ is continuous except at the inaccessible jump times of the Poisson measure, and at time $T$ with $\Delta \xi_T^{t,x} \leq 0$ a.s., 
and this implies   the continuity of $A^{t,x}$ by Th. 2.6 in \cite{17}.
Moreover, 
 Th. 3.2 in \cite{17} ensures that
\begin{equation}
Y_t^{t,x} = 
{\rm ess} \sup_{\tau \in \mathcal{T}_t} \mathcal{E}_{t, \tau}^{t,x}(\xi_\tau^{t,x}) \quad \text{a.s.}
\end{equation}
The SDE \eqref{SDE} and the RBSDE \eqref{markRBSDE} can be solved with respect to the translated 
Brownian motion $(W_s - W_t)_{s \geq t}$.  Hence $Y_t^{t,x}$ is constant for each $t,x$. We can thus define 
a deterministic function $u$ called {\em value function} of our optimal stopping problem by setting for each $t,x$
\begin{equation}\label{DEF}
u(t,x):=  Y_t^{t,x}.
\end{equation}
 By Lemma $\ref{conti}$ and Lemma $\ref{polyn}$ given in Appendix, the  function $u$ is continuous and has at most polynomial growth. \\ The continuity of  $u$ implies that $Y_s^{t,x}= u(s, X_s^{t,x})$, $t \leq s \leq T$ a.s.

Moreover, the stopping time $\tau^{*,t,x}$ (also denoted by $\tau^*$), defined by
$$\tau^{*}:= \inf \{ s \geq t,\,\, Y_s^{t,x} = \xi_s^{t,x}\}=  \inf \{ s \geq t,\,\, u(s, X_s^{t,x})=     \bar h(s, X_s^{t,x})\}$$
is an optimal stopping  time  for \eqref{vs}  
(see Th. 3.6  in \cite{17}).
Here, the function $\bar h$ is defined by $ \bar h (t,x) :=h(t,x) {\bf 1}_{t<T} + g(x) {\bf 1}_{t=T} $, so that $\xi_s^{t,x} = \bar h(s, X_s^{t,x})$, $0 \leq t \leq T$ a.s.

In the next section, we prove  that the value function is a viscosity solution of an obstacle problem.

\section{The Value Function, Viscosity Solution of an Obstacle Problem}\label{sec3}


We  consider the following related obstacle problem for a parabolic PIDE:
\begin{equation}\label{4.8}
\begin{cases}
 \min(u(t,x)-h(t,x),  -\dfrac{\p u}{\p t}(t,x)-Lu(t,x)-f(t,x,u(t,x), (\sigma \dfrac{\p u }{\p x})(t,x),Bu(t,x) )=0, \\ \quad (t,x) \in [0,T[ \times \mathbb{R}\\
u(T,x)=g(x), \;  x\in \mathbb{R}
\end{cases}
\end{equation}
where
\begin{align}
&L :=A+K, \nonumber \\
&A\phi(t,x) := \dfrac{1}{2}\sigma^2(x)\dfrac{\p^2 \phi}{\p x^2}(t,x)+ b(x) \dfrac{\p \phi }{\p x}(t,x), \nonumber  \\
&K \phi(t,x) :=\int_{\mathbb{R}^*}\left(\phi(t,x+ \beta(x,e))-\phi(t,x)- \dfrac{\p \phi}{\p x}(t,x)\beta(x,e)\right) \nu (de),    \label{defK3}  \\
&B \phi(t,x)(\cdot) :=\phi(t,x+\beta(x,\cdot))-\phi(t,x)  \in {L}_{\nu}^2.  \nonumber 
\end{align}
The operator  $B$ and $K$  are well defined for $\phi \in  C^{1,2}([0,T] \times \mathbb R)$. Indeed, since $\beta$ is bounded, we have 
$ |\phi(t,x+\beta(x,e))-\phi(t,x)| \leq C |\beta(x,e)|$ and
$$|\phi(t,x + \beta(x,e))-\phi(t,x)- \dfrac{\p \phi}{\p x}(t,x)\beta(x,e)| \leq C \beta(x,e)^2.$$
%

We prove below  that the value function $u$ defined by $\eqref{DEF}$ is a  viscosity solution of the above obstacle problem.
\begin{definition}\rm
$\bullet$ A continuous function $u$ is said to be a {\em viscosity subsolution} of \eqref{4.8} iff $u(T,x)\leq g(x), x\in \mathbb{R}$, and iff for any point $(t_0,x_0) \in (0,T) \times \mathbb{R}$ and for any $\phi \in  C^{1,2}([0,T] \times \mathbb{R})$ such that $\phi(t_0,x_0)=u(t_0,x_0)$ and $\phi-u$ attains its minimum at $(t_0,x_0)$, we have
\begin{align*}
&\min(u(t_0,x_0)-h(t_0,x_0),\\ &-\dfrac{\p \phi}{\p t}(t_0,x_0)-L\phi(t_0,x_0)-f(t_0,x_0,u(t_0,x_0), (\sigma \dfrac{\p \phi}{\p x})(t_0,x_0),B \phi(t_0,x_0)) \leq 0.
\end{align*}
In other words, if $u(t_0,x_0)>h(t_0,x_0)$, then 
\begin{equation*}
-\dfrac{\p \phi}{\p t} (t_0,x_0)-L\phi(t_0,x_0)-f(t_0,x_0,u(t_0,x_0),  (\sigma \dfrac{\p \phi}{\p x} )(t_0,x_0),B \phi(t_0,x_0)) \leq 0.
\end{equation*}

$\bullet$
 A continuous function $u$  is said to be a {\em  viscosity supersolution } of \eqref{4.8} iff $u(T,x)\geq g(x), x\in \mathbb{R}$, and  iff for any point $(t_0,x_0) \in (0,T) \times \mathbb{R}$ and for any $\phi \in C^{1,2}([0,T] \times \mathbb{R})$ such that $\phi(t_0,x_0)=u(t_0,x_0)$ and $\phi-u$ attains its maximum at $(t_0,x_0)$, we have
\begin{align*}
&\min(u(t_0,x_0)-h(t_0,x_0),\\& -\dfrac{\p \phi}{\p t} (t_0,x_0)-L \phi(t_0,x_0)-f(t_0,x_0,u(t_0,x_0),  (\sigma \dfrac{\p \phi}{\p x})(t_0,x_0),B \phi (t_0,x_0)) \geq 0.
\end{align*}
In other words, we have both $u(t_0,x_0) \geq h(t_0,x_0)$, and
\begin{equation*}
-\dfrac{\p \phi}{\p t}(t_0,x_0)-L \phi(t_0,x_0)-f(t_0,x_0,u(t_0,x_0), (\sigma \dfrac{\p \phi}{\p x} )(t_0,x_0),B \phi(t_0,x_0)) \geq 0.
\end{equation*}

\end{definition}

\begin{theorem}
The function $u$, defined by \eqref{DEF}, is a viscosity solution (i.e. both a viscosity sub- and supersolution)  of the  obstacle problem~\eqref{4.8}.

\end{theorem}

\begin{proof}

$\bullet$    We first prove that $u$  is a subsolution of \eqref{4.8}.

Let $(t_0,x_0) \in (0,T) \times \mathbb{R}$ and $\phi \in  C^{1,2}([0,T] \times \mathbb{R})$ be such that\linebreak $\phi(t_0, x_0)=u(t_0,x_0)$  and $\phi(t,x) \geq u(t,x)$, $\forall (t,x) \in [0,T] \times \mathbb{R}$.
Suppose by contradiction that $u(t_0,x_0)>h(t_0,x_0)$ and that
\begin{equation*}
 -\dfrac{\p \phi}{\p t} (t_0,x_0)-L\phi(t_0,x_0)-f(t_0,x_0, \phi (t_0,x_0), (\sigma \dfrac{\p \phi}{\p x})(t_0,x_0),B \phi(t_0,x_0))>0.
\end{equation*}

By continuity of $K \phi$ (which  can be shown using Lebesgue's theorem) and that of $B \phi:[0,T] \times \mathbb{R} \rightarrow {L}_{\nu}^2$, we can suppose that there exists $\vp>0$ and $\eta_\vp>0$ such that:
$\forall (t,x)$ such that $t_0 \leq t \leq t_0+ \eta_\vp<T$ and $|x-x_0| \leq \eta_\vp$, we have: $u(t,x) \geq h(t,x) + \vp$ and
\begin{equation}\label{4.9}
-\dfrac{\p \phi}{\p t}(t,x)-L \phi(t,x)-f(t,x, \phi (t,x), (\sigma \dfrac{\p \phi}{\p x})(t,x),B \phi(t,x)) \geq\vp.
\end{equation}
Note that $Y_s^{t_0,x_0}=Y_s^{s, X_s^{t_0,x_0}}=u(s,X_s^{t_0,x_0})$ a.s. because $X^{t_0,x_0}$ is a Markov process and $u$ is continuous.
We define the stopping time $\theta$ as:
\begin{equation}\label{4.10}
\theta:= (t_0 + \eta_\vp) \wedge \inf\{s \geq t_0, |X_s^{t_0,x_0}-x_0|> \eta_\vp \}.
\end{equation}
By definition of the stopping time $\theta$, \\
\begin{equation*}
u(s,X_s^{t_0,x_0}) \geq h(s,X_s^{t_0,x_0})+ \vp> h(s,X_s^{t_0,x_0}), t_0 \leq s < \theta \text{  a.s. }
\end{equation*}
This means that for a.e. $\omega$ the process $(Y_s^{t_0,x_0}(\omega), s \in [t_0, \theta(\omega)[)$ stays strictly above the barrier. It follows that  for a.e. $\omega$, the function   $s \rightarrow A_s^c(\omega)$ is constant on $[t_0, \theta(\omega)]$. In other words, $Y_s^{t_0,x_0}=\mathcal{E}_{s, \theta}^{t_0,x_0}(Y_{\theta})$, $t_0 \leq s \leq \theta$ $a.s$,   that  is $(Y_s^{t_0,x_0}, s \in [t_0, \theta])$ is the solution of the classical BSDE associated with  driver $f$, terminal time $\theta$ and  terminal value $Y_{\theta}^{t_0,x_0}$.
Applying It\^o's lemma to $\phi(t, X_t^{t_0,x_0})$, we get:
\begin{align}\label{4.11}
\phi(t, X_t^{t_0,x_0})&=\phi(\theta, X_{\theta}^{t_0,x_0}) -\int_{t}^{\theta}\psi(s,X_s^{t_0,x_0})ds 
- \int_{t}^{\theta}(\sigma \dfrac{\p \phi}{\p x})(s,X_s^{t_0,x_0})dW_s \nonumber\\
& \qquad -\int_{t}^{\theta} \int_{\mathbb{R^*}}B\phi(s,X_{s^-}^{t_0,x_0})\Tilde{N}(ds,de)
\end{align}
where
$\psi(s,x):=\dfrac{\p \phi}{\p s}(s,x)+L\phi(s,x).$\\
Note that $(\phi(s,X_s^{t_0,x_0}), (\sigma \dfrac{\p \phi }{\p x})(s, X_s^{t_0,x_0}), B\phi(s,X_{s^-}^{t_0,x_0}); s \in[t_0,\theta])$
 is the  solution of the BSDE associated to terminal time $\theta$, terminal value $\phi(\theta, X_{\theta}^{t_0,x_0})$ and driver process $-\psi(s,X_s^{t_0,x_0})$.

By \eqref{4.9} and the definition of the stopping time $\theta$, we have a.s. that for each  $ s \in [t_0,\theta]$:
\begin{align}\label{4.12}
&-\dfrac{\p \phi}{\p t}(s,X_s^{t_0,x_0})-L\phi(s,X_s^{t_0,x_0})  \nonumber \\
& \qquad -f\left(s,X_s^{t_0,x_0},\phi(s,X_s^{t_0,x_0}), (\sigma \dfrac{\p \phi}{\p x} )(s,X_s^{t_0,x_0}),B \phi(s,X_s^{t_0,x_0})\right) \geq\vp.
\end{align}
Using the definition of the function $\psi$, \eqref{4.12} can be  rewritten:  for all $s \in [t_0,\theta]$, 
 \begin{align*}
&-\psi(s,X_s^{t_0,x_0}) \\ & \quad -f \left(s,X_s^{t_0,x_0},\phi(s,X_s^{t_0,x_0}),(\sigma \dfrac{\p \phi}{\p x})(s,X_s^{t_0,x_0}),B \phi(s,X_s^{t_0,x_0})\right) \geq \vp.
\end{align*}
This  gives a relation between the drivers $-\psi(s,X_s^{t_0,x_0})$ and 
$f (s,X_s^{t_0,x_0}, \cdot)$
of the two BSDEs. Also, $\phi(\theta, X_{\theta}^{t_0,x_0}) \geq u(\theta, X_{\theta}^{t_0,x_0})=Y_{\theta}^{t_0,x_0}$ a.s. \\Consequently, the extended comparison result for  BSDEs with jumps given in the Appendix (see Proposition \ref{A.4}) implies that:
\begin{equation*}
\phi(t_0,x_0)=\phi(t_0,X_{t_0}^{t_0,x_0})> Y_{t_0}^{t_0,x_0}=u(t_0,x_0),
\end{equation*}
which leads to a contradiction.

\bigskip

$\bullet$   We now prove that $u$  is a viscosity supersolution of \eqref{4.8}.

Let $(t_0,x_0) \in (0,T) \times \mathbb{R}$ and $\phi \in  { C}^{1,2}([0,T] \times \mathbb{R})$ be such that\\ $\phi(t_0, x_0)=u(t_0,x_0)$  and $\phi(t,x) \leq u(t,x)$, $\forall (t,x) \in [0,T] \times \mathbb{R}$.
Since the solution $(Y_s^{t_0,x_0})$ stays above the obstacle, we have:
\begin{equation*}
u(t_0,x_0) \geq h(t_0,x_0).
\end{equation*}

We must prove that:
\begin{equation*}
-\dfrac{\p \phi}{\p t}(t_0,x_0)-L\phi(t_0,x_0)-f \left(t_0,x_0, \phi (t_0,x_0), (\sigma \dfrac{\p \phi}{\p x})(t_0,x_0),B \phi(t_0,x_0)\right) \geq 0.
\end{equation*}

Suppose by contradiction that:
\begin{equation*}
-\dfrac{\p \phi}{\p t}(t_0,x_0)-L\phi(t_0,x_0)-f\left(t_0,x_0, \phi (t_0,x_0), (\sigma \dfrac{\p \phi}{\p x})(t_0,x_0),B\phi(t_0,x_0)\right) < 0.
\end{equation*}

By continuity, we can suppose that there exists $\vp>0$ and $\eta_\vp>0$ such that for each
$(t,x)$ such that $t_0 \leq t \leq t_0+ \eta_\vp<T$ and $|x-x_0| \leq \eta_\vp$, we have:
\begin{equation}\label{4.14}
 -\dfrac{\p \phi}{\p t} (t,x)-L \phi(t,x)-f\left(t,x,\phi(t,x), (\sigma \dfrac{\p \phi }{\p x})(t,x),B \phi(t,x)\right) \leq -\vp.
\end{equation}
We define  the stopping time $\theta$ as:
\begin{equation*}
\theta:= (t_0 + \eta_\vp) \wedge \inf\{s \geq t_0/ |X_s^{t_0,x_0}-x_0|> \eta_\vp \}.
\end{equation*}
Applying as above It\^o's lemma to $\phi(s, X_s^{t_0,x_0})$,  we get that \\
$(\phi(s,X_s^{t_0,x_0}), (\sigma \dfrac{\p \phi}{\p x})(s, X_s^{t_0,x_0}), B\phi(s,X_{s^-}^{t_0,x_0}); s \in [t_0, \theta])$ is the solution of the BSDE associated with terminal value $\phi(\theta, X_{\theta}^{t_0,x_0})$ and driver $-\psi(s,X_s^{t_0,x_0})$.

The process $(Y^{t_0,x_0}, s \in [t_0, \theta])$ is the solution of the classical BSDE  associated with terminal condition $Y_{\theta}^{t_0,x_0}=u(\theta,X_{\theta}^{t_0,x_0})$ and  generalized driver
\begin{equation*}
f(s,X_s^{t_0,x_0},y,z,q)ds+dA_s^{t_0,x_0}.
\end{equation*}
By \eqref{4.14} and the definition of the stopping time $\theta$, we have :
\begin{align*}
&(-\dfrac{\p \phi}{\p t} (s,X_s^{t_0,x_0})-L \phi(s,X_s^{t_0,x_0})-f(s,X_s^{t_0,x_0},\phi (s,X_s^{t_0,x_0}),\\&\qquad  (\sigma \dfrac{\p  \phi}{\p x})(s,X_s^{t_0,x_0}),B \phi(s,X_s^{t_0,x_0})))ds-dA_s^{t_0,x_0} \leq -\vp \text{ }ds,\quad  t_0 \leq s \leq \theta \text{ a.s.}
\end{align*}
or, equivalently,
\begin{align*}
& -\psi(s,X_s^{t_0,x_0})ds \\
& \quad   \leq (f(s,X_s^{t_0,x_0},\phi (s,X_s^{t_0,x_0}),  (\sigma \dfrac{\p \phi}{\p x} )(s,X_s^{t_0,x_0}),B \phi (s,X_s^{t_0,x_0})))ds   \\
& \qquad +dA_s^{t_0,x_0}-\vp \text{ } ds, \quad t_0 \leq s \leq \theta \text{ a.s.}
\end{align*}
This gives a relation between the drivers of the two BSDEs.\\ Also, $\phi(\theta, X_{\theta}^{t_0,x_0}) \leq u(\theta,X_{\theta}^{t_0,x_0})=Y_{\theta}^{t_0,x_0}$ a.s.
Consequently, Proposition \ref{A.4}  in the Appendix implies that:
\begin{equation*}
\phi(t_0,x_0)=\phi(t_0,X_{t_0}^{t_0,x_0})< Y_{t_0}^{t_0,x_0}=u(t_0,x_0),
\end{equation*}
which leads to a contradiction. \qed
\end{proof}

%
%
%
\section{Uniqueness  Result for the Obstacle Problem}\label{sec4}

We provide a uniqueness result for \eqref{4.8} in the particular case when for each $\phi \in C^{1,2}([0,T] \times  \mathbb{R})$,   $B \phi$ is a map  valued in $\mathbb{R}$ instead of ${L}^2_{\nu}$. More precisely,
\begin{equation}\label{defB} 
 B \phi(t,x) :=\int_{\mathbb{R}^*}(\phi(t,x+\beta(x,e))-\phi(t,x))\gamma(x,e)\nu(de),
 \end{equation}
 which is well defined  since
$ |\phi(t,x+\beta(x,e))-\phi(t,x)| \leq C |\beta(x,e)|.$ \\ 
 We suppose that Assumption  $\textbf{H}_1$ holds and  we make the additional  assumptions:

{\bf Assumption $\textbf{H}_2$:}
$$ {1.}   \; f(s,X_s^{t,x}(\omega),y,z,k) :=\overline{f}\left(s,X_s^{t,x}(\omega),y,z,\int_{\mathbb{R}^*}k(e)\gamma(X_s^{t,x}(\omega),e)\nu(de)\right)\textbf{1}_{s \geq t},$$
where $\overline{f}: [0,T] \times  \mathbb{R}^4  \rightarrow \mathbb{R}$
is continuous in $t$ uniformly with respect to $x,y,z,k$,  continuous in $x$ uniformly  with respect to $y,z,k$, 
and satisfies:
\begin{itemize}

\item[(i)]
 $|\overline{f}(t,x,0,0,0)| \leq C,$ for all $t \in [0,T], x\in\mathbb{R}.$

\item[(ii)]
$ |\overline{f}(t,x,y,z,k)- \overline{f}(t,x',y',z',k')| \leq C(|y-y'|+|z-z'|+|k-k'|)$, for all $ t \in  [0,T]$, $y, y', z,z',k,k' \in \mathbb{R}$.

\item[(iii)]
$ k \mapsto \overline{f}(t,x,y,z,k)$ is non-decreasing, for all $ t \in  [0,T]$, $x,y,z \in \mathbb{R}$.
\end{itemize}
$2$.  For each $R>0$, there exists a continuous function $m_R: \mathbb{R}_{+} \rightarrow \mathbb{R}_+$   such that
$m_R(0)=0$ and $|\overline{f}(t,x,v,p,q) - \overline{f}(t,y,v,p,q)| \leq m_{R}(|x-y|(1+|p|)),$ \\  for all $ t \in  [0,T]$, $|x|, | y|\leq R, |v|\leq R,  \;p,q \in \mathbb{R}.$ \\
$3$.
$|\g(x,e) - \g(y,e)| \leq C|x-y|(1\wedge e^{2})$ and $0 \leq \gamma(x,e) \leq  C(1 \wedge |e|)$, for all $x,y \in \mathbb{R}, e \in \mathbb{R}^*.$ \\
$4$.
\text{  There exists }$r>0$ such that for all $t\in [0,T]$, $x,u,v,p,l\in \mathbb{R}$:
$$
\overline{f}(t,x,v,p,l) - \overline{f}(t,x,u,p,l)\geq  r(u-v)
\text{ when } u\geq v.$$
$5$.
$ |h(t,x)|+|g(x)| \leq C $, for all $ t \in  [0,T]$,    $x \in \mathbb{R}$. 

To simplify notation, $\overline{f}$ is denoted by $f$ in the sequel.

\

We state below a comparison theorem, which uses results of three lemmas. The proofs of these 
lemmas are given in Subsection \ref{PL}. 

\begin{theorem}[Comparison principle]\label{8.9} Under the above hypotheses, if $U$ is a viscosity subsolution  and $V$ is a viscosity supersolution of the obstacle problem \eqref{4.8}   in the class of continuous bounded  functions, then $U(t,x) \leq V(t,x)$,
for each $(t,x) \in [0,T] \times \mathbb{R}$.
\end{theorem}

\begin{proof}
Set
$$M:=\sup_{  [0,T] \times \mathbb{R} }(U-V).$$ 
It is sufficient to prove that $M\leq 0$. 
For each  $\vp, \eta >0 $, 
we introduce the function:
\begin{equation*}
\psi^{\vp, \eta}(t,s,x,y):=U(t,x)-V(s,y)- \dfrac{(x-y)^2}{\vp^2}-\dfrac{(t-s)^2}{\vp^2}- \eta^2(x^2+y^2),
\end{equation*} 
for $t,s,x,y$ in $[0,T]^2 \times \mathbb{R}^2$. 
Let $$M^{\vp, \eta}:=\max_{[0,T]^2 \times \mathbb{R}^2
 }\psi^{\vp, \eta}.$$
This supremum
is reached at some point $(t^{\vp, \eta},s^{\vp, \eta},x^{\vp, \eta},y^{\vp, \eta}).$

%
%
%
Using that $\psi^{\vp, \eta}(t^{\vp, \eta},s^{\vp, \eta},x^{\vp, \eta},y^{\vp, \eta}) \geq \psi^{\vp, \eta}(0,0,0,0)$, we obtain:
\begin{align}\label{4.21}
& U(t^{\vp, \eta},x^{\vp, \eta})-V(s^{\vp, \eta},y^{\vp, \eta})-\dfrac{(t^{\vp, \eta}-s^{\vp, \eta})^2}{\vp^2}-\dfrac{(x^{\vp, \eta}-y^{\vp, \eta})^2}{\vp^2} -\eta^2((x^{\vp, \eta})^2+(y^{\vp, \eta})^2) \geq U(0,0)-V(0,0),
\end{align}
or, equivalently,
\begin{align}\label{4.22}
&\dfrac{(t^{\vp, \eta}-s^{\vp, \eta})^2}{\vp^2}+\dfrac{(x^{\vp, \eta}-y^{\vp, \eta})^2}{\vp^2}+\eta^2((x^{\vp, \eta})^2+(y^{\vp, \eta})^2) \leq \|U\|_{\infty}+\|V\|_{\infty}-U(0,0)-V(0,0).
\end{align}
Consequently, we can find a constant $C$ such that:
\begin{align}\label{4.23}
|x^{\vp, \eta}-y^{\vp, \eta}|+|t^{\vp, \eta}-s^{\vp, \eta}| \leq C\vp\\
|x^{\vp, \eta}|\leq \dfrac{C}{\eta}, |y^{\vp, \eta}| \leq \dfrac{C}{\eta}.\label{4.23b}
\end{align}
Extracting a subsequence if necessary, we may suppose that for each $\eta$ the sequences $(t^{\vp, \eta})_{\vp}$ and  $(s^{\vp, \eta})_{\vp}$ converge to a common limit $t^{\eta}$ when 
$\vp$ tends to $0$, and from \eqref{4.23} and \eqref{4.23b} we may also suppose, extracting again, that for each $\eta$, the sequences $(x^{\vp, \eta})_{\vp}$ and $(y^{\vp, \eta})_{\vp}$ converge to a common limit $x^{\eta}.$
\begin{lemma}\label{converg}
We have:
\begin{align*} 
\lim_{\vp \rightarrow 0} \frac{(x^{\vp, \eta} - y^{\vp, \eta})^{2}}{\vp^{2}}=0\,; \quad \lim_{\vp \rightarrow 0}\frac{(t^{\vp, \eta}-s^{\vp, \eta})^{2}}{\vp^{2}}=0 \\
\lim_{\eta \rightarrow 0} \lim_{\vp \rightarrow 0} M^{\vp, \eta}=M.
 \end{align*}
\end{lemma}
We now introduce the functions:
\begin{equation*}
\Psi_1(t,x):=V(s^{\vp,\eta}, y^{\vp, \eta})+ \dfrac{(x-y^{\vp, \eta})^2}{\vp^2}+\dfrac{(t-s^{\vp, \eta})^2}{\vp^2}+ \eta^2(x^2+(y^{\vp, \eta})^2)
\end{equation*}
\begin{equation*}
\Psi_2(s,y):=U(t^{\vp,\eta}, x^{\vp, \eta})- \dfrac{(x^{\vp, \eta}-y)^2}{\vp^2}-\dfrac{(t^{\vp, \eta}-s)^2}{\vp^2}- \eta^2((x^{\vp, \eta})^2+y^2).
\end{equation*}

As $(t,x) \rightarrow  (U-\Psi_1)(t,x)$ reaches its maximum at $(t^{\vp, \eta},x^{\vp, \eta})$ and $U$ is a subsolution we have  two  cases:

$\bullet$ $t^{\vp,\eta}=T$ and then $U(t^{\vp, \eta}, x^{\vp, \eta}) \leq g(x^{\vp, \eta})$, \\

$\bullet$ $t^{\vp,\eta} \not =T$ and then
\begin{align}\label{4.16}
&\min \left(U(t^{\vp, \eta}, x^{\vp, \eta})-h(t^{\vp, \eta}, x^{\vp, \eta}), \;\;
 \dfrac{\p \Psi_1}{\p t}(t^{\vp, \eta},x^{\vp, \eta})-L \Psi_1(t^{\vp, \eta},x^{\vp, \eta}) -  \right. \nonumber \\
 & \left.  -  f\left(t^{\vp, \eta},x^{\vp, \eta},U(t^{\vp, \eta},x^{\vp, \eta}), (\sigma \dfrac{\p \Psi_{1}}{\p x} )(t^{\vp, \eta},x^{\vp, \eta}), B \Psi_1(t^{\vp, \eta},x^{\vp, \eta})\right)\right) \leq 0.
\end{align}
As $(s,y) \rightarrow  (\Psi_2-V)(s,y)$ reaches its maximum at $(s^{\vp, \eta},y^{\vp, \eta})$ and $V$ is a supersolution we have the two following cases:
\begin{itemize}
\item[$\bullet$]
 $s^{\vp,\eta}=T$ and then $V(s^{\vp, \eta}, y^{\vp, \eta}) \geq g(y^{\vp, \eta})$,
\item[$\bullet$]
 $s^{\vp,\eta} \not =T$ and then
\begin{align}\label{4.17}
&\min (V(s^{\vp, \eta}, y^{\vp, \eta})-h(s^{\vp, \eta}, y^{\vp, \eta}),\nonumber\\& \dfrac{\p \Psi_2}{\p t}(s^{\vp, \eta},y^{\vp, \eta})-L \Psi_2(s^{\vp, \eta},y^{\vp, \eta})-f(s^{\vp, \eta},y^{\vp, \eta},V(s^{\vp, \eta},y^{\vp, \eta}),\\&\qquad (\sigma \dfrac{\p \Psi_{2} }{\p x} )(s^{\vp, \eta},y^{\vp, \eta}), B \Psi_2(s^{\vp, \eta},y^{\vp, \eta})) \geq 0.\nonumber
\end{align}
\end{itemize}
We now prove that 
$
M \leq 0.
$
Three cases are possible. \\
\textbf{1st case:} There exists a subsequence of $(t^{\eta})$ such that $t^{\eta}=T$ for all $\eta$ (of this subsequence).
As $U$ is continuous, for all $\eta$ and for  $\vp$ small enough
\begin{equation*}
U(t^{\vp, \eta},x^{\vp, \eta}) \leq U(t^{\eta},x^{ \eta})+ \eta \leq  g(x^{\eta})+ \eta,
\end{equation*}
and as $V$ is continuous, for all $\eta$ and for $\vp$ small enough
\begin{equation*}
V(s^{\vp, \eta},y^{\vp, \eta}) \geq V(t^{\eta},x^{ \eta})- \eta \geq g(x^{\eta})- \eta.
\end{equation*}
Hence
\begin{equation*}
U(t^{\vp, \eta},x^{\vp, \eta})- V(s^{\vp, \eta},y^{\vp, \eta}) \leq 2\eta
\end{equation*}
and
\begin{align*}
M^{\vp, \eta}&=U(t^{\vp, \eta},x^{\vp, \eta})- V(s^{\vp, \eta},y^{\vp, \eta})-\dfrac{(x^{\vp, \eta}-y^{\vp, \eta})^2}{\vp^2}-\dfrac{(t^{\vp, \eta}-s^{\vp, \eta})^2}{\vp^2}\\&-\eta^2((x^{\vp, \eta})^2+(y^{\vp, \eta})^2) \leq U(t^{\vp, \eta},x^{\vp, \eta})- V(s^{\vp, \eta},y^{\vp, \eta}) \leq 2\eta.
\end{align*}
Letting $\vp \rightarrow 0$ and then $\eta \rightarrow 0$ one gets, using Lemma \ref{converg}, that $M \leq 0$.

\textbf{2nd case:} There exists a subsequence such that $t^{\eta} \not= T$, and for all $\eta$ belonging to this subsequence, there exists a subsequence of $(x^{\vp,\eta})_{\eta}$ such that
\begin{equation*}
U(t^{\vp, \eta},x^{\vp, \eta})-h(t^{\vp, \eta},x^{\vp, \eta}) \leq 0.
\end{equation*}
As from \eqref{4.17} one has
\begin{equation*}
V(s^{\vp, \eta},y^{\vp, \eta})-h(s^{\vp, \eta},y^{\vp, \eta}) \geq 0,
\end{equation*}
it comes that
\begin{equation*}
M^{\vp, \eta} \leq  U(t^{\vp, \eta},x^{\vp, \eta})- V(s^{\vp, \eta},y^{\vp, \eta}) \leq h(t^{\vp, \eta},x^{\vp, \eta})-h(s^{\vp, \eta},y^{\vp, \eta}).
\end{equation*}
Letting $\vp \rightarrow 0$ and then $\eta \rightarrow 0$, using the equality $\lim_{\eta \rightarrow 0} \lim_{\vp \rightarrow 0} M^{\vp, \eta}=M$ (see Lemma \ref{converg}), we derive that $M \leq 0$. \\

\textbf{Last case:} We are left with the case when, for a subsequence of $\eta$, we have $t^{\eta} \not= T$ and for all $\eta$ belonging to this subsequence there exists a subsequence of $(x^{\vp, \eta})_{\vp}$ such that:
\begin{equation*}
U(t^{\vp,\eta},x^{\vp, \eta})-h(t^{\vp, \eta}, x^{\vp, \eta})>0.
\end{equation*}
Set
\begin{equation}\label{4.24}
\v(t,s,x,y) := \frac{(x-y)^{2}}{\vp^{2}} + \frac{(t-s)^{2}}{\vp^{2}} + \eta^{2}(x^{2} + y^{2}).
\end{equation}
The maximum of the function $\psi^{\vp,\eta}(t,s,x,y): = U(t,x) - V(s,y)- \v(t,s,x,y)$ is reached at the point $(t^{\vp,\eta}, s^{\vp,\eta}, x^{\vp,\eta}, y^{\vp,\eta})$. We apply the non-local version of Jensen Ishii's lemma \cite{1} and we obtain that there exist:
\begin{equation*}
(a,\overline{p},X)\in \cal{P}^{2,+}U(t^{\vp,\eta}, x^{\vp,\eta}),\; (b,\overline{q},Y)\in \cal{P}^{2,-}V(s^{\vp,\eta}, y^{\vp,\eta})
\end{equation*}
such that
\begin{equation*}
\begin{cases}
\overline{p} = p+2\eta^{2}x^{\vp,\eta}; \; \;    \overline{q} = p-2\eta^{2}y^{\vp,\eta}; \;\; 
p = \frac{2(x^{\vp,\eta} - y^{\vp,\eta})}{\vp^{2}}\\a =b= \frac{2(t^{\vp,\eta} - s^{\vp,\eta})}{\vp^{2}}\\
\begin{pmatrix}
X&0\\0&-Y
\end{pmatrix} \leq \frac{2}{\vp^{2}}\begin{pmatrix}
1&-1\\-1&1
\end{pmatrix} +2\eta^{2}\begin{pmatrix}
1&0\\0&1
\end{pmatrix}.
\end{cases}
\end{equation*}
Here, $\mathcal{P}^{2,+}$ (resp. $\mathcal{P}^{2,-}$) is the set of superjets (resp. subjets) defined in  \cite{1} (see Definition 3). 
Since $(t^{\vp,\eta}, s^{\vp,\eta}, x^{\vp,\eta}, y^{\vp,\eta})$ is a global maximum of $\psi^{\vp,\eta}$ ,we have:
\begin{align*}
&\psi^{\vp,\eta}(t^{\vp,\eta}, s^{\vp,\eta}, x^{\vp,\eta}+\b (x^{\vp,\eta}, e), y^{\vp,\eta}+\b(y^{\vp,\eta},e))\leq \psi^{\vp,\eta}(t^{\vp,\eta}, s^{\vp,\eta}, x^{\vp,\eta}, y^{\vp,\eta})
\nonumber\\
&\qquad \Leftrightarrow U(t^{\vp,\eta}, x^{\vp,\eta}+\b(x^{\vp,\eta},e)) - V(s^{\vp,\eta}, y^{\vp,\eta}+\b(y^{\vp,\eta},e))\\
&\qquad-\frac{(x^{\vp,\eta}+\b(x^{\vp,\eta},e) - y^{\vp,\eta} - \b(y^{\vp,\eta}, e))^{2}}{\vp^{2}}  \nonumber\\
&\qquad- \frac{(t^{\vp,\eta} - s^{\vp,\eta})^{2}}{\vp^{2}} - \eta^{2}((x^{\vp,\eta} +\b (x^{\vp,\eta}, e))^{2}+(y^{\vp,\eta}+\b(y^{\vp,\eta},e))^{2})\\&\qquad\leq U(t^{\vp,\eta}, x^{\vp,\eta}) - V(s^{\vp,\eta}, y^{\vp,\eta}) \nonumber\\
&\qquad- \frac{(x^{\vp,\eta} - y^{\vp,\eta})^{2}}{\vp^{2}}  - \frac{(t^{\vp,\eta}-s^{\vp,\eta})^{2}}{\vp^{2}} - \eta^{2}((x^{\vp,\eta})^{2} + (y^{\vp,\eta})^{2}).
\end{align*}
Consequently, we get:
\begin{align}\label{4.27}
&U(t^{\vp,\eta}, x^{\vp,\eta} +\b (x^{\vp,\eta}, e)) - U(t^{\vp,\eta}, x^{\vp,\eta}) \leq V(s^{\vp,\eta}, y^{\vp,\eta}+\b(y^{\vp,\eta},e)) \nonumber\\
&- V(s^{\vp,\eta}, y^{\vp,\eta})+
 \frac{(\b(x^{\vp,\eta},e) - \b(y^{\vp,\eta},e))^{2}}{\vp^{2}}
+p(\b(x^{\vp,\eta},e) - \b(y^{\vp,\eta},e)) \nonumber \\ 
&+\eta^{2}(\b^{2}(x^{\vp,\eta},e)+ 2x^{\vp,\eta}\b(x^{\vp,\eta}, e) + 2y^{\vp,\eta}\b(y^{\vp,\eta},e) + \b^{2}(y^{\vp,\eta},e)).
\end{align}
Let us fix $\delta >0$ and consider the ball  $\mathcal{B}_{\delta}= \mathcal{B}(0,\delta)$. 
We introduce the  operators  $K^\d$, $ \Tilde{K}^{\d}, {B}^{\d},  \Tilde{B}^{\d}$ 
 corresponding to the operators $K$  and $B$  defined in \eqref{defK3}  and  \eqref{defB}, but integrating  on $\mathcal{B}_\delta$ or $\mathbb{R}\backslash \mathcal{B}_\delta$ (also denoted by $\mathcal{B}^{c}_\delta$) only.\\
They are defined respectively  for all $\phi \in  C^{1,2}$,   $\Phi \in { \cal C}$ by 
\begin{align}
&K^{\d}[t, x, \phi]  :=\int_{\mathcal{B}_\d}\left(\phi( t, x +\b(x, e)) 
\dd-\phi( t, x) - \dfrac{\p \phi }{\p x}( t, x)\b(x, e)\right)\nu(de)   \label{estim1} \\
& \Tilde{K}^{\d}[t, x, \pi,\Phi] := \int_{\mathcal{B}^{c}_\d} \bigg(\Phi(t, x+\b(x, e)) - \Phi(t, x) - \pi \b(x,e)\bigg) \nu(de).\label{estim2} 
\end{align}
\begin{align}
 {B}^{\d}[t,x,\phi]&:=\int_{\mathcal{B}_\d}\bigg(\phi( t, x+\b(x,e))   -  \phi(t,x)\bigg)\g(x,e)\nu(de) \label{estim5} \\
 \Tilde{B}^{\d}[t, x, \Phi] &:= \int_{\mathcal{B}^c_\d}\bigg(\Phi(t, x+\b(x,e)) - \Phi(t, x)\bigg)\g(x,e)\nu(de) \label{estim6} 
 \end{align}
 Here ${ \cal C}$ denotes the set of bounded continuous functions. 

  
  Note that the operators ${K}^{\d}$,   $\Tilde{K}^{\d}$, ${B}^{\d}$ and $\Tilde{B}^{\d}$ satisfy the hypotheses (NLT) of \cite{1} (see Section 2.2 in   \cite{1}). Hence we can use the alternative definition for sub-superviscosity solutions in terms of sub-superjets (see Definition 4 in \cite{1}). 
Since $U$ is a subviscosity solution and $V$ is superviscosity solution,  we have:
\begin{equation}\label{4.25}
\begin{cases}
F(t^{\vp,\eta},x^{\vp,\eta}, U(t^{\vp,\eta},x^{\vp,\eta}), a, \overline{p}, X, K^{\d}[t^{\vp,\eta}, x^{\vp,\eta}, \v_{x}]  \\\qquad +
\Tilde{K}^{\d}[t^{\vp,\eta}, x^{\vp,\eta}, \overline{p}, U], {B}^{\d}[t^{\vp,\eta}, x^{\vp,\eta}, \v_{x}]  + \Tilde{B}^{\d}[t^{\vp,\eta},x^{\vp,\eta},U])\leq 0\\F(s^{\vp,\eta}, y^{\vp,\eta}, V(s^{\vp,\eta}, y^{\vp,\eta}), a, \overline{q}, Y,  {K}^{\d}[s^{\vp,\eta}, y^{\vp,\eta}, - \v_{y}] \\\qquad+ \Tilde{K}^{\d} [s^{\vp,\eta}, y^{\vp,\eta}, \overline{q}, V], {B}^{\d}[s^{\vp,\eta}, y^{\vp,\eta}, -\v_{y}]+ \Tilde{B}^{\d}[s^{\vp,\eta}, y^{\vp,\eta}, V])\geq 0
\end{cases}
\end{equation}
where
\begin{equation}\label{4.26}
F(t,x,u,a,p,X, l_{1},l_{2}): = -a-\frac{1}{2}\sigma^2 (x)X - b(x)p-l_{1} -f(t,x,u,p\s(x),l_{2}).
\end{equation}
 We denote by $\v_{x}$ the function $(t,x) \mapsto \v(t, x, s^{\vp,\eta}, y^{\vp,\eta})$ and by $\v_{y}$ the function $(s,y) \mapsto \v(t^{\vp,\eta},x^{\vp,\eta},s, y)$. The two following lemmas hold.
 
\begin{lemma}\label{Lestim}
Let
\begin{align}\label{4.28}
l_{K} &:= {K}^{\d}[t^{\vp,\eta}, x^{\vp,\eta},\v_{x}] + \Tilde{K}^{\d}[t^{\vp,\eta},x^{\vp,\eta},\overline{p},U] \nonumber\\
l_{K}^{'} &:= {K}^{\d}[s^{\vp,\eta}, y^{\vp,\eta}, -\v_{y}] + \Tilde{K}^{\d}[s^{\vp,\eta}, y^{\vp,\eta}, \overline{q}, V].
\end{align}
We have
\begin{equation}\label{estim}
l_{K} \leq l_{K}^{\prime}+O(\frac{(x^{\vp,\eta}-y^{\vp,\eta})^{2}}{\vp^{2}}) + O(\eta^{2}) + (\frac{1}{\vp^{2}}+\eta^{2})O(\d).
\end{equation}
\end{lemma}
\begin{lemma}\label{Lestim1}
Let
\begin{align}\label{4.32}
l_{B}:= {B}^{\d}[t^{\vp,\eta}, x^{\vp,\eta}, \v_{x}]+\Tilde{B}^{\d}[t^{\vp,\eta}, x^{\vp,\eta},U] \nonumber \\
l_{B}^{\prime}:= {B}^{\d}[s^{\vp,\eta}, y^{\vp,\eta},-\v_{y}] + \Tilde{B}^{\d}[s^{\vp,\eta},y^{\vp,\eta},V].
\end{align}
We have
\begin{equation}\label{4.37}
 l_{B}\leq l_{B}^{\prime}+(\eta^{2} + \frac{1}{\vp^{2}})O(\d) + O(\frac{(x^{\vp,\eta}-y^{\vp,\eta})^{2}}{\vp^{2}}) +  O(|x^{\vp,\eta}-y^{\vp,\eta}|)+O(\eta^{2}).
\end{equation}
\end{lemma}
\bigskip
We argue now by contradiction by assuming that 
\begin{equation}\label{absurde}
M>0.
\end{equation}
Using  Assumption  $(\textbf{H}_2).4$,
 we get
\begin{align}
& 
0<\frac{r}{2}M\leq 
r M_{\vp,\eta} \leq r(U(t^{\vp,\eta}, x^{\vp,\eta}) - V(s^{\vp,\eta}, y^{\vp,\eta}))\nonumber\\
&\quad \leq F(s^{\vp,\eta}, y^{\vp,\eta}, U(t^{\vp,\eta}, x^{\vp,\eta}), a, \overline{q}, Y, l_{K}^{\prime}, l_{B}^{\prime})   \nonumber\\
&\qquad- F(s^{\vp,\eta}, y^{\vp,\eta}, V(s^{\vp,\eta}, y^{\vp,\eta}), a, \overline{q}, Y, l_{K}^{\prime}, l_{B}^{\prime})\nonumber\\
&\quad = F(s^{\vp,\eta}, y^{\vp,\eta}, U(t^{\vp,\eta}, x^{\vp,\eta}), a, \overline{q}, Y, l_{K}^{\prime}, l_{B}^{\prime})\nonumber\\
&\qquad  -F(s^{\vp,\eta}, y^{\vp,\eta}, U(s^{\vp,\eta}, y^{\vp,\eta}), a, \overline{q}, Y, l_{K}^{\prime}, l_{B}^{\prime}) \nonumber\\
&\qquad + F(s^{\vp,\eta}, y^{\vp,\eta}, U(s^{\vp,\eta}, y^{\vp,\eta}), a, \overline{q}, Y, l_{K}^{\prime}, l_{B}^{\prime}) \nonumber\\
&\qquad - F(s^{\vp,\eta}, y^{\vp,\eta}, U(s^{\vp,\eta}, y^{\vp,\eta}), a, \overline{q}, Y, l_{K}, l_{B}) \nonumber\\
&\qquad + F(s^{\vp,\eta}, y^{\vp,\eta}, U(s^{\vp,\eta}, y^{\vp,\eta}), a, \overline{q}, Y, l_{K}, l_{B}) \nonumber\\
&\qquad - F(t^{\vp,\eta}, x^{\vp,\eta}, U(t^{\vp,\eta}, x^{\vp,\eta}), a, \overline{p}, X, l_{K}, l_{B}) \nonumber\\
&\qquad +F(t^{\vp,\eta}, x^{\vp,\eta}, U(t^{\vp,\eta}, x^{\vp,\eta}), a, \overline{p}, X, l_{K}, l_{B})  \nonumber\\
&\qquad -F(s^{\vp,\eta}, y^{\vp,\eta}, V(s^{\vp,\eta}, y^{\vp,\eta}), a, \overline{q}, Y, l_{K}^{\prime}, l^{\prime}_{B}) \nonumber\\
&\;\;\leq K|U(t^{\vp,\eta}, x^{\vp,\eta}) - U(s^{\vp,\eta}, y^{\vp,\eta})| 
 + F(s^{\vp,\eta}, y^{\vp,\eta}, U(s^{\vp,\eta}, y^{\vp,\eta}), a, \overline{q}, Y, l_{K}, l_{B}) \nonumber\\
 &\qquad - F(t^{\vp,\eta}, x^{\vp,\eta}, U(t^{\vp,\eta}, X^{\vp,\eta}), a, \overline{p}, X, l_{K}, l_{B})\nonumber\\
 &\qquad + (\eta^{2}+\frac{1}{\vp^{2}})O(\d) + O(\frac{(x^{\vp,\eta}-y^{\vp,\eta})^{2}}{\vp^{2}}) + O(|x^{\vp,\eta} - y^{\vp,\eta}|)+O(\eta^{2}). \label{4.38}
\end{align}
We have used here the (nonlocal) ellipticity of $F$, the Lipschitz property of $F$, (\ref{4.25})  and the estimates proven in Lemma \ref{Lestim} and Lemma \ref{Lestim1}.
From the hypothesis on $b$ and $\sigma$, we have:
\begin{align*}
& \s^{2}(x^{\vp,\eta})X-\s^{2}(y^{\vp,\eta})Y \leq \frac{C(x^{\vp,\eta}-y^{\vp,\eta})^{2}}{\vp^{2}} + O(\eta^{2}),\\&b(x^{\vp,\eta})\overline{p} - b(y^{\vp,\eta})\overline{q}\leq \frac{C|x^{\vp,\eta} - y^{\vp,\eta}|}{\vp^{2}} + O(\eta^{2}).
\end{align*}
We thus obtain the  inequality:
\begin{align}
&F(s^{\vp,\eta}, y^{\vp,\eta}, U(s^{\vp,\eta}, y^{\vp,\eta}), a, \overline{q}, Y, l_{K},l_{B}) - F(t^{\vp,\eta}, x^{\vp,\eta}, U(t^{\vp,\eta}, x^{\vp,\eta}), a, \overline{p}, X, l_{K},l_{B})\nonumber \\
&\qquad \leq \frac{C(x^{\vp,\eta}-y^{\vp,\eta})^{2}}{\vp^{2}} + O(\eta^{2}) \nonumber \\
&\qquad + f(t^{\vp,\eta}, x^{\vp,\eta}, U(t^{\vp,\eta}, x^{\vp,\eta}), (p+2\eta^{2})\s(x^{\vp,\eta}),l_{B}) \nonumber  \\
&\qquad -f(s^{\vp,\eta}, y^{\vp,\eta}, U(s^{\vp,\eta}, y^{\vp,\eta}), (p-2\eta^{2})\s(y^{\vp,\eta}),l_{B})\nonumber\\
&\qquad \leq  f(t^{\vp,\eta}, x^{\vp,\eta}, U(t^{\vp,\eta}, x^{\vp,\eta}), (p+2\eta^{2})\s(x^{\vp,\eta}),l_{B}) \nonumber \\
&\qquad  -f(s^{\vp,\eta}, x^{\vp,\eta}, U(t^{\vp,\eta}, x^{\vp,\eta}), (p+2\eta^{2})\s(x^{\vp,\eta}),l_{B}) \nonumber \\
&\qquad+ m_{R}(|x^{\vp,\eta}-y^{\vp,\eta}|(1+(p+2\eta^{2})\s(x^{\vp,\eta})))\nonumber\\
&\qquad  + {K}|U(t^{\vp,\eta}, x^{\vp,\eta}) - U(s^{\vp,\eta}, y^{\vp,\eta})|+O(\frac{(x^{\vp,\eta}-y^{\vp,\eta})^{2}}{\vp^{2}})+O(\eta^{2}). \label{4.39}
\end{align}
The last equality is obtained by some computations  similar to those in \eqref{4.38}.
From (\ref{4.38}), (\ref{4.39}) we get
\begin{align}
& 0<\frac{r}{2}M\leq 
r M^{\vp,\eta}\leq f(t^{\vp,\eta}, x^{\vp,\eta}, U(t^{\vp,\eta}, x^{\vp,\eta}), (p+2\eta^{2})\s(x^{\vp,\eta}),l_{B})  \nonumber \\
&\qquad  -f(s^{\vp,\eta}, x^{\vp,\eta}, U(t^{\vp,\eta}, x^{\vp,\eta}), (p+2\eta^{2})\s(x^{\vp,\eta}),l_{B})  \nonumber \\
&\qquad+ m_{R}(|x^{\vp,\eta}-y^{\vp,\eta}|(1+(p+2\eta^{2})\s(x^{\vp,\eta})) \nonumber\\
&\qquad +K|U(t^{\vp,\eta}, x^{\vp,\eta})-U(s^{\vp,\eta}, y^{\vp,\eta})|+  \nonumber\\
&\qquad +O(\frac{(x^{\vp,\eta}-y^{\vp,\eta})^{2}}{\vp^{2}}) + O(|x^{\vp,\eta}-y^{\vp,\eta}|) + (\eta^{2}+\frac{1}{\vp^{2}})O(\d)+O(\eta^{2}). \label{4.40}
\end{align}
By Lemma \ref{converg}, letting successively $\d,\vp$ and $\eta$ tend to $0$ in  (\ref{4.40}) we obtain that 
$0<\frac{r}{2}M\leq 0$. 
Hence, the assumption $M>0$ made above (see \eqref{absurde}) is wrong.
This  ends the proof of Theorem \ref{8.9}. 
\end{proof}

\begin{coro}[Uniqueness] 
Under the  additional Assumption  $(\textbf{H}_2)$, the value function is  the unique solution of the  obstacle problem $(\ref{4.8})$ in the class of bounded continuous  functions.
\end{coro}


\subsection{Proofs  of the  lemmas }\label{PL}

\paragraph{Proof of  Lemma \ref{converg}.}
For $\eta >0$, we introduce the functions:\\
$\Tilde{U}^{\eta}(t,x)=U(t,x)-\eta^2x^2$ and $\Tilde{V}^{\eta}(t,x)=V(t,x)+\eta^2 x^2$. 
Set
$$M^\eta:=\sup_{  [0,T] \times \mathbb{R} }(\Tilde{U}^{\eta}-\Tilde{V}^{\eta}).$$ 
The maximum $M^\eta$ is reached at some point $(\hat{t}^\eta, \hat{x}^\eta)$. 
From the form of  $\psi^{\vp, \eta}$, we have that for fixed $\eta$, 
 there exists a subsequence $(t^{\vp, \eta},s^{\vp, \eta},x^{\vp \eta},y^{\vp, \eta})_\vp$ which converges to some point $(t^\eta, s^\eta, x^\eta, y^\eta)$ when $\vp$ tends to $0$.

Since $M^{\vp, \eta}$ is reached at $(t^{\vp, \eta},s^{\vp, \eta},x^{\vp, \eta},y^{\vp, \eta})$, we have:
\begin{align*}
&(\Tilde{U}^\eta-\Tilde{V}^\eta)(\hat{t}^\eta, \hat{x}^\eta)= (U-V)(\hat{t}^\eta, \hat{x}^\eta)-\eta^2((\hat{x}^\eta)^2+(\hat{y}^\eta)^2)  \leq M^{\vp, \eta} 
\\ & 
 =  U(t^{\vp, \eta},x^{\vp, \eta})-V(s^{\vp, \eta},y^{\vp, \eta})  -\dfrac{(t^{\vp, \eta}-s^{\vp, \eta})^2}{\vp^2}\\ & \qquad  -\dfrac{(x^{\vp, \eta}-y^{\vp, \eta})^2}{\vp^2}-\eta^2((x^{\vp, \eta})^2+(y^{\vp, \eta})^2).
\end{align*}
Setting
$$\overline l_{\eta}:=\lim \sup_{\vp \rightarrow 0} \frac{(x^{\vp, \eta} - y^{\vp, \eta})^{2}}{\vp^{2}}\,,\quad
\underline l_{\eta}:=\lim \inf_{\vp \rightarrow 0} \frac{(x^{\vp, \eta} - y^{\vp, \eta})^{2}}{\vp^{2}}$$
we get
\begin{align}\label{lim1}
 0 \leq \underline l_{\eta} \leq \overline l_{\eta}  \leq (\Tilde{U}^\eta-\Tilde{V}^\eta)(t^{\eta},x^{\eta})-(\Tilde{U}^\eta-\Tilde{V}^\eta)(\hat{t}^\eta, \hat{x}^\eta) \leq 0.
\end{align}
We derive that, up to a subsequence, $\lim_{\vp \rightarrow 0} \frac{(x^{\vp, \eta} - y^{\vp, \eta})^{2}}{\vp^{2}}= 0$ and\\ $\lim_{\vp \rightarrow 0}M^{\vp, \eta}=M^{\eta}.$
Similarly, we get $\lim_{\vp \rightarrow 0} \frac{(t^{\vp, \eta} - s^{\vp, \eta})^{2}}{\vp^{2}}= 0$.

Let us prove that $\lim_{\eta \rightarrow 0}M^{\eta} = M$.
First, note that $M^\eta \leq M,$ for all $\eta$.\\
By definition of $M$, for all $\delta>0$ there exists $(t_\delta, x_\delta ) \in [0,T] \times \mathbb{R}$ such that\\
$M-\delta \leq (U-V)(t_\delta, x_\delta).$
Consequently, we get $$M-2\eta^2x_\delta^2-\delta \leq (U-V)(t_\delta, x_\delta)-2\eta^2x_\delta^2=(\Tilde{U}^\eta-\Tilde{V}^\eta)(t_\delta, x_\delta) \leq M^\eta \leq M.$$
By letting $\eta$ and then $\delta$ tend to $0$, the result follows.  \qed 
\paragraph{ Proof of   Lemma \ref{Lestim}.}
We have:
\begin{align}\label{op1}
{K}^{\d}[t^{\vp,\eta},x^{\vp,\eta},\v_{x}] =\dd\int_{\mathcal{B}_\d}(\frac{1}{\vp^{2}}+\eta^{2})\b^{2}(x^{\vp,\eta},e)\nu(de)\\
\label{op2}
{K}^{\d}[s^{\vp,\eta}, y^{\vp,\eta}, -\v_{y}] = \dd\int_{\mathcal{B}_\d} (-\frac{1}{\vp^{2}}-\eta^{2})\b^{2}(y^{\vp,\eta}, e)\nu(de).
\end{align}
Equations $\eqref{op1}$ and $\eqref{op2}$ imply:
\begin{align}
&\dd {K}^{\d}[t^{\vp,\eta}, x^{\vp,\eta}, \v_{x}]{\leq} {K}^{\d}[s^{\vp,\eta}, y^{\vp,\eta}, -\v_{y}] {+} (\frac{1}{\vp^{2}}{+}\eta^{2})\int_{\mathcal{B}_\d}\b^{2}(y^{\vp,\eta}, e)\nu(de) \nonumber \\
&\dd + (\frac{1}{\vp^{2}}{+}\eta^{2})\int_{\mathcal{B}_\d}\b^{2}(x^{\vp,\eta},e)\nu(de) 
 {\leq} {K}^{\d}[s^{\vp,\eta}, y^{\vp,\eta}, -\v_{y}] {+}(\frac{1}{\vp^{2}}{+}\eta^{2})O(\d).
\label{4.29}
\end{align}
Using  inequality \eqref{4.27} and integrating on $\mathcal{B}_\delta^c$, we obtain:
\begin{align*}
&\Tilde{K}^{\d}[t^{\vp,\eta}, x^{\vp,\eta}, \overline{p},U] = \int_{\mathcal{B}^{c}_\d}\bigg(U(t^{\vp,\eta}, x^{\vp,\eta}+\b (x^{\vp,\eta},e))-U(t^{\vp,\eta}, x^{\vp,\eta})\\
& - (p+2\eta^{2}x^{\vp,\eta}) \b(x^{\vp,\eta}, e)\bigg)\nu(de) \leq \int_{\mathcal{B}^{c}_\d}\bigg(V(s^{\vp,\eta}, y^{\vp,\eta}+\b(y^{\vp,\eta},e)) -V(s^{\vp,\eta}, y^{\vp,\eta})\\
&- (p-2\eta^{2}y^{\vp,\eta})\b(y^{\vp,\eta}, e)\bigg)\nu(de)  + \int_{\mathcal{B}^{c}_\d}\frac{(\b(x^{\vp,\eta},e) - \b(y^{\vp,\eta}, e))^{2}}{\vp^{2}}\nu(de)\\
&+\eta^{2}\int_{\mathcal{B}^{c}_\d}(\b^{2}(x^{\vp,\eta},e) +\b^{2}(y^{\vp,\eta},e))\nu(de)\nonumber\\
&\leq \Tilde{K}^{\d}[s^{\vp,\eta}, y^{\vp,\eta},\overline{q}, V] + O(\frac{(x^{\vp,\eta}-y^{\vp,\eta})^{2}}{\vp^{2}}) +O(\eta^{2}).
\end{align*}
Using \eqref{4.28} and \eqref{4.29}, we derive  \eqref{estim},
which ends the proof of Lemma \ref{Lestim}.
\qed

\paragraph{ Proof of   Lemma \ref{Lestim1}.}
From \eqref{estim5},  
 we derive that:
\begin{align}\label{e1}
 {B}^{\d}[t^{\vp,\eta}, x^{\vp,\eta}, \v_{x}]& = \int_{{\mathcal{B}_\d}}\bigg((\eta^{2}+\frac{1}{\vp^{2}}) \b^{2}(x^{\vp,\eta},e) + \frac{2\b(x^{\vp,\eta},e)}{\vp^{2}}(x^{\vp,\eta}-y^{\vp,\eta})\nonumber\\&+2\eta^{2}x^{\vp, \eta}\b(x^{\vp,\eta},e)\bigg) \g(x^{\vp,\eta},e)\nu(de) \\
{B}^{\d}[s^{\vp,\eta}, y^{\vp,\eta},-\v_{y}]& = \int_{\mathcal{B}_\d}\bigg((-\eta^{2}-\frac{1}{\vp^{2}}) \b^{2}(y^{\vp,\eta},e)+ \frac{2\b(y^{\vp,\eta},e)}{\vp^{2}}(x^{\vp,\eta}-y^{\vp,\eta})\nonumber\\&-2\eta^{2}y^{\vp,\eta}\b(y^{\vp,\eta},e)\bigg) \g(y^{\vp,\eta},e)\nu(de). \label{e2}
\end{align}
After some computations, we obtain:
\begin{align}
& \bigg((\eta^{2}+\frac{1}{\vp^{2}})\b^{2}(x^{\vp,\eta},e) + \frac{2\b(x^{\vp,\eta},e)}{\vp^{2}} (x^{\vp,\eta}-y^{\vp,\eta}) +
 2\eta^{2}x^{\vp,\eta}\b(x^{\vp,\eta},e)\bigg)\g(x^{\vp,\eta},e)\nonumber\\
 & = (-\eta^{2}-\frac{1}{\vp^{2}}) \b^{2}(y^{\vp,\eta},e)\g(y^{\vp,\eta},e) +\frac{2\b(y^{\vp,\eta},e)}{\vp^{2}}(x^{\vp,\eta}-y^{\vp,\eta})\g(y^{\vp,\eta},e) \nonumber\\
 &- 2\eta^{2}y^{\vp,\eta}\b(y^{\vp,\eta},e)\g(y^{\vp,\eta},e)  \nonumber \\
 & + (\eta^{2}+\frac{1}{\vp^{2}})\bigg(\b^{2}(y^{\vp,\eta},e) \g(y^{\vp,\eta},e) + \b^{2}(x^{\vp,\eta},e)\g(x^{\vp,\eta},e)\bigg)\nonumber\\
 & +\frac{2}{\vp^{2}}(x^{\vp,\eta}-y^{\vp,\eta})\bigg(\b(x^{\vp,\eta},e)\g(x^{\vp,\eta},e) - \b(y^{\vp,\eta},e)\g(y^{\vp,\eta},e)\bigg)\nonumber\\
 & +2\eta^{2}\bigg(x^{\vp,\eta}\b(x^{\vp,\eta},e) \g(x^{\vp,\eta},e)+y^{\vp,\eta}\b(y^{\vp,\eta},e) \g(y^{\vp,\eta},e)\bigg). \label{e3}
\end{align}
From \eqref{e1}, \eqref{e2}, \eqref{e3} and using the hypothesis on $\b$ and $\g$, we get:
\begin{equation}\label{4.36}
{B}^{\d}[t^{\vp,\eta}, x^{\vp,\eta}, \v_{x}]\leq {B}^{\d}[s^{\vp,\eta},y^{\vp,\eta},-\v_{y}]+ (\eta^{2}+\frac{1}{\vp^{2}}) O(\d) + O(\frac{(x^{\vp,\eta}-y^{\vp,\eta})^{2}}{\vp^{2}}) +O(\eta^{2}).
\end{equation}
We now estimate  the operator $\Tilde{B}^{\delta}$. Inequality (\ref{4.27}) implies:
\begin{align*}
&\bigg(U(t^{\vp,\eta}, x^{\vp,\eta}+\b(x^{\vp,\eta},e)) - U(t^{\vp,\eta},x^{\vp,\eta})\bigg)\g(x^{\vp,\eta},e)\\
&\qquad \leq \bigg(V(s^{\vp,\eta}, y^{\vp,\eta}+\b(y^{\vp,\eta},e)) - V(s^{\vp,\eta},y^{\vp,\eta})  \\
&\qquad +\frac{|\b(x^{\vp,\eta},e) - \b(y^{\vp,\eta},e)|^{2}}{\vp^{2}} +p( \b(x^{\vp,\eta},e)-\b(y^{\vp,\eta},e)) \\
&\qquad +\eta^{2}(\b^{2}(x^{\vp,\eta}, e)+2x^{\vp,\eta}\b(x^{\vp,\eta},e)+2y^{\vp,\eta}\b(y^{\vp,\eta},e)+\b^{2}(y^{\vp,\eta},e)\bigg)\g(x^{\vp,\eta},e)\\
&\qquad =\bigg(V(s^{\vp,\eta}, y^{\vp,\eta}+\b(y^{\vp,\eta},e)) - V(s^{\vp,\eta}, y^{\vp,\eta})\bigg)\g(y^{\vp,\eta}, e)\\
&\qquad + \bigg(V(s^{\vp,\eta},y^{\vp,\eta}+\b(y^{\vp,\eta},e)) - V(s^{\vp,\eta},y^{\vp,\eta})\bigg)\bigg(\g(x^{\vp,\eta},e) -\g(y^{\vp,\eta},e)\bigg) \\
&\qquad  
+\frac{|\b(x^{\vp,\eta},e)-\b(y^{\vp,\eta},e)|^{2}}{\vp^{2}}\g(x^{\vp,\eta},e) + p  \bigg(\b(x^{\vp,\eta},e) - \b(y^{\vp,\eta},e)\bigg) \g(x^{\vp,\eta},e)\\
&\qquad + \eta^{2}\bigg(\b^{2}(x^{\vp,\eta},e) +2x^{\vp,\eta}\b(x^{\vp,\eta},e)+ 2y^{\vp,\eta}\b(y^{\vp,\eta},e) +\b^{2}(y^{\vp,\eta},e)\bigg) \g(x^{\vp,\eta},e).
\end{align*}
Now, by  \eqref{4.23b}, we have  $|x^{\vp, \eta}|\leq \dfrac{C}{\eta}$ and $|y^{\vp, \eta}| \leq \dfrac{C}{\eta}$.
Hence, using the hypothesis on $\beta, \gamma $ and integrating on 
$ \mathcal{B}_\delta^c$, we get
\begin{equation}\label{4.35}
\Tilde{B}^{\d}[t^{\vp,\eta}, x^{\vp,\eta}, U]\leq \Tilde{B}^{\d}[s^{\vp,\eta}, y^{\vp,\eta}, V] + O(|x^{\vp,\eta} - y^{\vp,\eta}|) + O(\frac{(x^{\vp,\eta}-y^{\vp,\eta})^{2}}{\vp^{2}})+O(\eta^{2}).
\end{equation}
Finally, from \eqref{4.36}, \eqref{4.32} and \eqref{4.35}, we derive inequality \eqref{4.37}.
 \qed

\section{Conclusions}
  In this paper, we  have  studied  the optimal stopping problem for a monotonous dynamic risk measure  defined by a Markovian BSDE with jumps. We have proven that,  under relatively weak hypotheses, the value function is a  viscosity solution  of an obstacle problem   for a
partial integro-differential variational inequality.
To obtain the uniqueness of the solution under appropriate conditions, 
we have proven a comparison theorem, based on the nonlocal version of the Jensen Ishii Lemma, which extends some results established in \cite{1} (Section 5.1, Th.3) to the case of a nonlinear BSDE.

 The  links given in this paper  between optimal stopping problems for BSDEs 
 and obstacle problems for PDEs 
 can be extended to a larger class of  problems. Among them, we can mention generalized Dynkin games with nonlinear expectation  (see \cite{DQS2}),  and  mixed optimal stopping/stochastic control problems (see \cite{DQS1}). However, the latter case  requires to establish  a dynamic 
programming principle,  which does not  follow from the flow property of reflected BSDEs only,  and 
needs rather sophisticated  techniques.

\appendix

\section{Appendix}
\subsection{Some Useful Estimates}

Let $T>0$ be a fixed terminal time.

A map $f: [0,T] \times \Omega  \times \mathbb{R}^2 \times  L_{\nu}^2  \rightarrow \mathbb{R}; (t, \omega, y,z,k) \mapsto f(t, \omega, y,z,k)$ is  said  to be a \textit{ Lipschitz driver } if it is predictable, uniformly Lipchitz with respect to $y,z,k$ and such that $f(t,0,0,0) \in \mathbb{H}^2.$
%
%
%
%

Let $\xi_t^1, \xi_t^2 \in \mathcal{S}^2$. Let $f^1, f^2$ be two admissible Lipschitz drivers with Lipchitz constant $C$.
For $i=1,2$, let $\mathcal{E}^{i}$ be the $f^{i}$-conditional expectation associated with driver $f^{i}$, and let ($Y_t^{i}$) be the adapted process defined for each $t \in [0,T]$,
\begin{align}\label{procV}
Y_t^{i}:={\rm ess} \sup_{\tau \in \mathcal{T}_t} \mathcal{E}^{i}_{t, \tau}(\xi^{i}_\tau).
\end{align}


\begin{proposition} \label{oubli}

For $s \in [0,T]$, denote $\overline{Y}_s=Y_s^1-Y_s^2$, $\overline{\xi}_s=\xi_s^1-\xi_s^2$ and\\ $\overline{f}_s= \sup_{y,z,k}|f^1(s,y,z,k)-f^2(s,y,z,k)|$. Let $\eta, \beta >0$ be such that $\beta\geq \dfrac{3}{\eta}+2C$ and $\eta \leq \dfrac{1}{C^2}$. Then for each $t$, we have:
\begin{equation}\label{eqA.1}
e^{\beta t} \overline{Y}_t^2  \leq e^{\beta T}(\mathbb{E}[\sup_{s \geq t} \overline{\xi_s}^2| \mathcal{F}_t]+ \eta \mathbb{E}[\int_t^T{\overline{f}_s^2}ds|\mathcal{F}_t]) \textsc{ } a.s.
\end{equation}

\end{proposition}


\begin{proof}
For $i=1,2$ and for each $\tau \in \mathcal{T}_{0}$, let $(X^{i,\tau}$, $\pi_s^{i,\tau}, l_s^{i, \tau})$  be the solution of the BSDE associated with driver $f^i$, terminal time $\tau$ and terminal condition $\xi_{\tau}^i$.
Set  $\overline{X}_s^{\tau}=X_s^{1,\tau}-X_s^{2,\tau}$.

By a priori estimate on BSDEs (see Proposition $A.4$ in \cite{17}), we have:
\begin{align}\label{A.2}
e^{\beta t} (\overline{X}_t^{\tau})^2 & \leq e^{\beta T} \mathbb{E}[\overline{\xi}_{\tau}^2|\mathcal{F}_t]+ \eta \mathbb{E}[\int_t^Te^{\beta s}(f^1(s, X_s^{2,\tau},\pi_s^{2,\tau}, l_s^{2,\tau})\nonumber\\&-f^2(s, X_s^{2,\tau},\pi_s^{2,\tau}, l_s^{2,\tau}))^2 ds| \mathcal{F}_t] \quad \text{ a.s. }
\end{align}
from which we derive that
\begin{equation}\label{A.3}
e^{\beta t} (\overline{X}_t^{\tau})^2  \leq e^{\beta T} (\mathbb{E}[\sup_{s \geq t}\overline{\xi}_{s}^2|\mathcal{F}_t]+ \eta \mathbb{E}[\int_t^T{\overline{f}_s^2 ds}| \mathcal{F}_t]).
\end{equation}
Now, by definition of $Y^{i}$, we have $Y_t^i ={\rm ess} \sup_{\tau \geq t} X_t^{i, \tau}$ a.s. for $i=1,2$. We thus  get $|\overline{Y}_t| \leq{\rm ess} \sup_{\tau \geq t}|\overline{X}_t^{\tau}|$ a.s.
The result follows.\qed \end{proof}

Let $\xi_t \in \mathcal{S}^2$. Let $f$ be a Lipschitz driver  with Lipschitz constant $C>0$.
Set
\begin{align}\label{procV}
Y_t:={\rm ess} \sup_{\tau \in \mathcal{T}_t} \mathcal{E}_{t, \tau}(\xi_\tau)
\end{align}
where $\mathcal{E}$ is the $f$-conditional expectation associated with driver $f$.
\begin{proposition} \label{A.4}
 Let $\eta, \beta >0$ be such that $\beta\geq \dfrac{3}{\eta}+2C$ and $\eta \leq \dfrac{1}{C^2}$. Then for each $t$, we have:
\begin{equation}
e^{\beta t} Y_t^2  \leq e^{\beta T}(\mathbb{E}[\sup_{s \geq t} {\xi_s}^2| \mathcal{F}_t]+ \eta \mathbb{E}[\int_t^T{f(s,0,0,0)^2}ds|\mathcal{F}_t]) \textsc{ } a.s.
\end{equation}
\end{proposition}
\begin{proof}
Let $X_t^{\tau}$ be the solution of the BSDE associated with driver $f$, terminal time $\tau$ and terminal condition $\xi_{\tau}$. By applying inequality \eqref{A.2} with $f^1=f$, $\xi_1=\xi$, $f^2=0$ and $\xi^2=0$, we get:
\begin{equation}\label{A.5}
e^{\beta t} (X_t^{\tau})^2 \leq e^{\beta T} \mathbb{E} [\xi_{\tau}^2|\mathcal{F}_t]+ \eta \mathbb{E} [\int_t^T{e^{\beta s}(f(s,0,0,0))^2}|\mathcal{F}_t].
\end{equation}
The result follows. \end{proof}
\begin{remark}
If the drivers satisfy Assumption 3.1  in \cite{17}, then $Y$ (resp. $Y^{i}$) is the solution of the RBSDE associated with driver $f$   (resp.$f^{i}$) and obstacle  $\xi$ (resp. $\xi^{i}$). Hence  the above estimates provide some new estimates on RBSDEs.  Note that $\eta$ and $\beta$ are universal constants, i.e. they do not depend on $T$, $\xi, \xi^1, \xi^2, f, f^1, f^2$. This was not the case for the estimates given in the previous literature (see e.g.  \cite{10}).
\end{remark}

\subsection{Some Properties of the Value Function $u$}

We prove below the continuity and polynomial growth of the function $u$ defined by (\ref{DEF}).
\begin{lemma}\label{conti}
The function $u$ is continuous in $(t,x)$.
\end{lemma}
\begin{proof}
It is sufficient to show that, when $(t_n, x_n) \rightarrow (t,x)$,
$
 |u(t_n, x_n)-u(t,x)| \rightarrow 0$. \\
Let $\bar h$ be the map defined by $ \bar h (t,x) =h(t,x)$ for $t<T$ and $\bar h (T,x) =g(x)$,
so that, for each $(t,x)$, we have
 $\xi^{t,x}_s = \bar h(s, X^{t,x}_s)$, $0 \leq s \leq T$ a.s.
 By applying  Proposition  \ref{oubli} with\linebreak  $X_s^1=X_s^{t_n,x_n}$, $X^2_s= X_s^{t,x}$, $f^1(s, \omega, y,z,q):= \textbf{1}_{[t,T]}(s) f(s, X_s^{t,x}(\omega),y,z,q)$ and\\   $f^2(s, \omega, y,z,q):= \textbf{1}_{[t_n,T]}(s) f(s, X_s^{t_n,x_n}(\omega),y,z,q)$,  we obtain:
$$
 |u(t_n, x_n)-u(t,x)|^2
 \leq K_{C,T} \mathbb{E}[\sup_{0\leq s\leq T}|\overline{h}(s, X_{s}^{t_n, x_n}) - \overline{h}(s, X_{s}^{t,x})|^2+ \int_{0}^{T}(\overline{f}_s^n)^2],$$ where
\begin{align*}
\begin{cases}
K_{C,T}:=e^{(3C^2+2C)T}\max(1, \dfrac{1}{C^2})\\ \overline{f}_s^n(\omega):=\sup_{y,z,q}|\textbf{1}_{[t,T]}f(s, X_s^{t,x}(\omega),y,z,q)-\textbf{1}_{[t_n,T]}f(s, X_s^{t_n,x_n}(\omega),y,z,q)|.
\end{cases}
\end{align*}
The continuity of $u$ is then a consequence of  the following convergences as $n\to \infty$:
\begin{align*}&\mathbb{E}(\sup_{0\leq s\leq T}|\overline{h}(s, X_{s}^{t,x})- \overline{h}(s, X_{s}^{t_n}( x_n))|^{2})\to 0\\
& \mathbb{E}[\int_{0}^{T}(\overline{f}_s^n)^{2}ds]\to 0,
\end{align*}
which follow from the Lebesgue's theorem, using the continuity assumptions and  polynomial growth of $f$ and $h$ . \qed
\end{proof}
\begin{lemma}\label{polyn}
The function $u$ has at most polynomial growth at infinity.
\end{lemma}
\begin{proof} By applying Prop. $\ref{A.4}$ , we obtain the following estimate:
\begin{equation}\label{4.7}
u(t,x)^2 \leq K_{C,T}(\mathbb{E}(\int_{0}^{T} f(s, X_{s}^{t,x},0,0,0)^{2}ds + \sup_{0\leq s\leq T}\overline{h}(s,X_{s}^{t,x})^{2}).
\end{equation}
Using now the hypothesis of polynomial growth on $f,h,g$  and the standard estimate
\begin{equation*}
\mathbb{E}[\sup_{0\leq s\leq T}|X_{s}^{t,x}|^{2}] \leq C'(1+x^{2}),
\end{equation*}
we derive that there exist $\bar C \in \mathbb{R}$ and $p\in \mathbb{N}$ such that $|u(t,x)| \leq \bar C(1+x^{p}),\ \forall t\in [0,T]$, $ \forall x\in \mathbb{R}.$ 
\qed
\end{proof}
\begin{remark}
By \eqref{4.7}, if $(t,x) \mapsto f(t,x, 0, 0), \; h$ and $g$ are bounded, then $u$ is bounded.
\end{remark}
\subsection{An Extension of the Comparison Result for BSDEs with Jumps}
We provide here an extension of the comparison theorem for BSDEs given in  \cite{16} which  formally
states that if two  drivers $f_1, f_2$ satisfy $f_1 \geq f_2 + \vp$, then 
the associated solutions $X^1$ and $X^2$ satisfy $X_0^1 >   X_0^2 $. 
\begin{proposition}\label{A.4}
Let $t_0 \in [0,T]$ and let $\theta$ be a stopping time such that $\theta > t_0$ a.s.\,\\
Let $\xi_1$ and $\xi_2$ $\in$ $L^2({\cal F}_{\theta})$.
Let $f_1$ be  a driver.
Let $f_2$ be a Lipschitz driver. 
For $i=1,2$, let $(X^i_t, \pi^i_t, l^i_t)$ be a solution in $S^{2} \times \H^{2} \times \H_{\nu}^{2}$ of the BSDE
\begin{equation}\label{eq7}
-dX^i_t  = \displaystyle   f_i (t, X^i_t, \pi^i_t, l^i_t) dt - \pi^i_t dW_t - \int_{\mathbb{R}^*} l^i_t(u) \Tilde{N} (dt,du); \quad X^i_{\theta}  = \xi_i.
\end{equation}
Assume that there exists a bounded predictable process $(\gamma_t)$ such that $dt\otimes dP \otimes \nu(de)$-a.s.\\$ \gamma_t(e) \geq -1  \;\; \text{ and }
\;\; |\gamma_t(e) | \leq C(1 \wedge |e|)$,
and such that 
\begin{equation}\label{autre}
f_2(t, X^2_t, \pi^2_t, l^1_t) - f_2(t,X^2_t,\pi^2_t,l^2_t) \geq \langle \gamma_t\,,\, l^1_t- l^2_t \rangle_\nu , \;\;  t_0 \leq  t \leq \theta,\; \; dt\otimes dP \text{ a.s.}
\end{equation}
Suppose also
 that
\begin{align*}
&\xi_1 \geq \xi_2   \text{ a.s. } \\
&f_1 (t, X^1_t, \pi^1_t, l^1_t) \geq f_2 (t, X^1_t, \pi^1_t, l^1_t) +\varepsilon, \;\; t_0 \leq t \leq \theta,\; \; dt\otimes dP \text{ a.s.}
\end{align*}
 where $\varepsilon $ is a real constant.  
 Then, 
 $$X_{t_0}^1-X_{t_0}^2 \geq \varepsilon \alpha  \quad {\rm a.s.}\,$$
 where $\alpha$ is a non negative ${\cal F}_{t_0}$-measurable r.v. which does not depend on $\varepsilon$, with 
 $P(\alpha >0) >0$.
  \end{proposition}
 
 \begin{proof}
 From inequality (4.22) in the proof of the Comparison Theorem in \cite{16}, we derive that 
 \begin{equation*}\label{eq9999}
  X_{t_0}^1-X_{t_0}^2 \geq  e^{-CT}\mathbb{E} \left[  \int_{t_0}^{\theta} H_{t_0,s}\,  \varepsilon \,ds | \mathcal{F}_{t_0} \right]\;
 \quad {\rm a.s.}\,,
\end{equation*}
where $C$ is the Lipschitz constant of $f_2$, and 
$(H_{t_0,s})_{s \in [t_0,T]}$ is the square integrable non negative martingale satisfying
\begin{equation*}
d H_{t_0,s}  = \displaystyle H_{t_0,s^-} \left[ \beta_s d W_s + \int_{\mathbb{R}^*} \gamma_s(u) \Tilde{N}(ds,du)\right] ; \;\; 
H_{t_0,t_0}  = 1,
\end{equation*}
 $(\beta_s)$ being  a predictable process bounded by $C$. We get
 $$ X_{t_0}^1-X_{t_0}^2 \geq  e^{-CT}\varepsilon \, \mathbb{E} \left[  H_{t_0,\theta}\, (\theta - t_0) 
    | \mathcal{F}_{t_0} \right]\;  \quad {\rm a.s.}\,$$
 Since $\theta > t_0$ a.s.\,, we have $ H_{t_0,\theta}\, (\theta - t_0)  \geq 0$ a.s. and 
 $P( H_{t_0,\theta}\, (\theta - t_0)  >0) >0$.
Setting $\alpha:= e^{-CT} \, \mathbb{E} \left[  H_{t_0,\theta}\, (\theta - t_0) 
    | \mathcal{F}_{t_0} \right]$, the result follows.
 \qed
 \end{proof}


\begin{thebibliography}{}

\bibitem{21} Bally, V., Caballero, M.E., Fernandez, B.,  El-Karoui N.: Reflected BSDE's, PDE's and Variational Inequalities, INRIA Research report, (2002)


\bibitem{2} Barles, G.,  Buckdahn R.,  Pardoux, E.: Backward stochastic differential equations and integral-partial differential equations, Stochastics and Stochastics Reports,  {\bf 60} (1-2), 57-83 (1997)

\bibitem{1}  Barles, G.,  Imbert, C.: Second-order elliptic integro-differential equations: viscosity solutions theory revisited, Ann. Inst. H. Poincar\'{e} Anal. Non Lin\'{e}aire, {\bf 25}, 567--585 (2008)


\bibitem{6}  Crandall, M.,  Ishii, H., Lions P-L.: User's guide to viscosity solutions of second order partial differential equations, American Mathematical Society, {\bf 27}, 1-67 (1992)

\bibitem{DQS1}  Dumitrescu, R., Quenez, M.-C., Sulem A.:  Dynamic programming principle for mixed optimal/stopping problems with $f$-conditional expectations, manuscript (2014)

\bibitem{DQS2} Dumitrescu, R., Quenez M.-C., Sulem A.:
Double barrier reflected BSDEs with jumps and generalized Dynkin games, arXiv:1310.2764 (2013)


\bibitem{7} Essaky, E.H., Reflected backward stochastic differential equation with jumps and RCLL obstacle, Bulletin des Sciences Math\'{e}matiques, {\bf 132}, 690--710 (2008)


\bibitem{10} El Karoui, N., Kapoudjian, C., Pardoux, E., Peng, S., Quenez, M-C.: Reflected solutions of backward SDE's, and related obstacle problems for PDE's, The Annals of Probability, {\bf 25}, 702-737 (1997)

\bibitem{8} Hamad\`{e}ne, S., Ouknine, Y.: Reflected backward stochastic differential equation with jumps and random obstacle, Electronic Journal of Probability,  {\bf 8}, 1--20 (2003)

\bibitem{9} Hamad\`{e}ne S., Ouknine, Y.: Reflected backward SDEs with general jumps,
 Manuscript (2007)


\bibitem{13} Ouknine, Y.: Reflected backward stochastic differential equation with jumps, Stochastic and Stoch. Reports, {\bf 65}, 111-125 (1998)

\bibitem{14} Pardoux E.,  Peng, S.:
Backward Stochastic Differential equations and Quasilinear Parabolic
Partial Differential equations, Lect. Notes in CIS, {\bf 176}, 200--217 (1992)


\bibitem{15}  Peng, S.:
Nonlinear expectations, nonlinear evaluations and risk measures,
165-253, Lecture Notes in Math., {\bf 1856}, Springer, Berlin  (2004)

\bibitem{16} Quenez, M.-C., Sulem A., BSDEs with jumps, optimization and applications to dynamic risk measures,
 Stoch. Proc. and Their Appl., {\bf123}, 3328--3357  (2013)


\bibitem{17} Quenez, M.-C., Sulem A.: Reflected BSDEs and robust optimal stopping for dynamic risk measures with jumps,  
Stoch. Proc. and Their Appl., {\bf 124}, 3031--3054 (2014)

\bibitem{18} Royer, M.: Backward stochastic differential equations with jumps and related non-linear expectations, Stoch. Proc. and Their Appl.,  {\bf 116}, 1358-1376 (2006)


\bibitem{20} Tang, S.H., Li X.: Necessary conditions for optimal control of stochastic systems with random jumps, SIAM J. Cont. and Optim.,  {\bf 32}, 1447--1475 (1994)

\end{thebibliography}
\end{document}